\documentclass[12pt]{amsart}
\usepackage{extsizes}
\usepackage{lipsum}
\usepackage{color}
\usepackage{amsfonts}
\usepackage{amssymb, latexsym, amsmath, amsthm}

\input{xypic}
\xyoption{all}

\newtheorem{thm}{Theorem}[section]
\newtheorem{lemma}[thm]{Lemma}
\newtheorem{prop}[thm]{Proposition}

\newtheorem{remark}[thm]{Remark}

\numberwithin{equation}{section}

\newcommand{\authorfootnotes}{\renewcommand\thefootnote{\@fnsymbol\c@footnote}}%
\makeatother

\begin{document}

\def\q{\mathfrak{q}}
\def\p{\mathfrak{p}}
\def\l{\mathfrak{l}}
\def\u{\mathfrak{u}}
\def\a{\mathfrak{a}}
\def\b{\mathfrak{b}}
\def\m{\mathfrak{m}}
\def\n{\mathfrak{n}}
\def\c{\mathfrak{c}}
\def\d{\mathfrak{d}}
\def\e{\mathfrak{e}}
\def\k{\mathfrak{k}}
\def\z{\mathfrak{z}}
\def\h{{\mathfrak h}}
\def\gl{\mathfrak{gl}}
\def\sl{\mathfrak{sl}}
\def\bk{k}
\def\bl{l}
\def\Ext{{\rm Ext}}
\def\Hom{{\rm Hom}}
\def\Ind{{\rm Ind}}

\def\res{\mathop{Res}}

\def\GL{{\rm GL}}
\def\SL{{\rm SL}}
\def\SO{{\rm SO}}
\def\O{{\rm O}}

\def\R{\mathbb{R}}
\def\C{\mathbb{C}}
\def\Z{\mathbb{Z}}
\def\Q{\mathbb{Q}}
\def\A{\mathbb{A}}
\def\bs{\boldsymbol}

\def\w{\wedge}

\def\Cat{\mathcal{C}}
\def\HC{{\rm HC}}
\def\HCat{\Cat^\HC}
\def\proj{{\rm proj}}

\def\to{\rightarrow}
\def\To{\longrightarrow}

\def\1{1\!\!1}
\def\dim{{\rm dim}}

\def\th{^{\rm th}}
\def\isom{\approx}

\def\CE{\mathcal{C}\mathcal{E}}
\def\E{\mathcal{E}}

\def\dis{\displaystyle}
\def\f{{\bf f}}                 
\def\g{{\bf g}}
\def\T{{\rm T}}              
\def\omegatil{\tilde{\omega}}  
\def\H{\mathcal{H}}            
\def\Dif{\mathfrak{D}}      
\def\W{W^{\circ}}           
\def\Whit{\mathcal{W}}      
\def\ringO{\mathcal{O}}     
\def\S{\mathcal{S}}      
\def\M{\mathcal{M}}      
\def\K{{\rm K}}          
\def\h{\mathfrak{h}} 
\def\N{\mathfrak{N}}    
\def\norm{{\rm N}}       
\def\trace{{\rm Tr}} 
\def\ctilde{\tilde{C}}

\author{Alia Hamieh}
\address[Alia Hamieh]{University of Lethbridge, Department of Mathematics and Computer Science, C526 University Hall, 4401 University Drive, Lethbridge, AB T1K3M4, Canada}
\email{alia.hamieh@uleth.ca}

\author{Naomi Tanabe}
\address[Naomi Tanabe]{Dartmouth College, Department of Mathematics, 6188 Kemeny Hall, Hanover, NH 03755-3551, USA}
\email{naomi.tanabe@dartmouth.edu}
\thanks{Research of both authors was partially supported by Coleman Postdoctoral Fellowships at Queen's University}
\thanks{Research of first author is currently supported by PIMS Postdoctoral Fellowship at the University of Lethbridge}

\title[Determining Hilbert Modular Forms: Weight Aspect]{Determining Hilbert Modular Forms by Central Values of Rankin-Selberg Convolutions: The Weight Aspect}

\date{\today}

\begin{abstract}
The purpose of this paper is to prove that a primitive Hilbert cusp form $\g$ is uniquely determined by the central values of the Rankin-Selberg $L$-functions $L(\f\otimes\g, \frac{1}{2})$, where $\f$ runs through all primitive Hilbert cusp forms of weight $k$ for infinitely many weight vectors $k$. This result is a generalization of the work of Ganguly, Hoffstein, and Sengupta \cite{ganguly-hoffstein-sengupta} to the setting of totally real number fields, and it is a weight aspect analogue of the authors own work \cite{2016TRANS}.
 \end{abstract}

\maketitle

\setcounter{section}{-1}
\section{Introduction}
Over the past five decades, a significant part of the research in analytic and algebraic number theory has been centered around the special values of $L$-functions. This will undoubtedly continue in view of the bountiful number of conjectures that this topic encompasses. Indeed, the special values of $L$-functions encode profound information about the underlying algebraic or geometric objects. For example, the vanishing or non-vanishing of $L$-functions and their twists at the center of the critical strip have rich arithmetic implications, as it is most famously portrayed by the Birch and Swinnerton-Dyer Conjecture and its far-reaching generalizations.  Another intriguing problem is to study the extent to which the special values of $L$-functions associated to abelian varieties or automorphic forms, for example, determine these structures. In 1996, Stark used transcendence theory to prove that the central value of $L(E,s)$, when non-vanishing, determines the isogeny class of the (modular) elliptic curve $E/\mathbb{Q}$. However, the elegant transcendental argument of Stark \cite{stark} does not seem to generalize in any obvious way to modular forms of higher weights. In an interesting paper \cite{luo-ramakrishnan}, Luo and Ramakrishnan achieved a breakthrough in this direction by showing that a primitive modular form $f$ is uniquely determined by the central values $L(f\otimes\chi,\frac12)$ as $\chi$ varies over a carefully chosen set of infinitely many quadratic Dirichlet characters. The method used in \cite{luo-ramakrishnan} sets a solid analytic framework which was subsequently used in several papers on the topic of determining modular forms by twists of central $L$-values such as Ganguly-Hoffstein-Sengupta \cite{ganguly-hoffstein-sengupta}, Luo \cite{luo2}, Pi \cite{Pi}, and Zhang \cite{zhang}. 

In a recent paper \cite{2016TRANS}, the authors generalized the result of Luo \cite{luo2} to the setting of totally real number fields by proving that a primitive Hilbert cusp form $\g$ is uniquely determined by the central values of the Rankin-Selberg $L$-functions $L(\f\otimes\g, \frac{1}{2})$, where $\f$ varies over all primitive Hilbert cusp forms of level $\q$ for infinitely many prime ideals $\q$. In the present paper, the authors employ the method of \cite{ganguly-hoffstein-sengupta} and some of the tricks used in \cite{2016TRANS} to establish a weight aspect analogue of their result \cite[Theorem~0.1]{2016TRANS}. This extends the main theorem in \cite{ganguly-hoffstein-sengupta} to the setting of totally real number fields. Indeed, it is the purpose of this paper to prove the following theorem. The reader is referred to Section~\ref{sec:background} for notation and terminology.

\begin{thm}\label{thm:main}
Let $\g\in S_l^{\mathrm{new}}(\n)$ and $\g'\in S_{l'}^{\mathrm{new}}(\n')$ be normalized Hecke eigenforms, where the weights $l$ and $l'$ are in $2\mathbb{N}^n$. If
\[ L\left(\f\otimes\g, \frac12\right)=L\left(\f\otimes\g', \frac12\right)\]
for all normalized Hecke eigenforms $\f\in S_k(\ringO_F)$ for infinitely many $k\in 2\mathbb{N}^n$, then $\g=\g'$.
\end{thm}
It should be noted that, when saying infinitely many $k=(k_1,\dots, k_n)$, we require $k_j$ to be sufficiently large for all $j$. The proof of this theorem is given in Section~\ref{sec:proof}. It follows as an application of an asymptotic formula for a weighted sum of the central values $L(\f\otimes\g,\frac12)$, where $\g$ is fixed and $\f$ varies as prescribed above. This formula, which is stated in Section~\ref{sec:moment}, allows for expressing the Fourier coefficient of $\g$ at a prime ideal $\p$ in terms of the central values $L(\f\otimes\g,\frac12)$ up to a negligible error term. Then, we see that two forms $\g$ and $\g'$ satisfying the hypotheses of the theorem above will ultimately have the same Fourier coefficients for all but finitely many prime ideals $\p$, in which case the strong multiplicity one theorem guarantees that the forms $\g$ and $\g'$ are necessarily equal.

The strategy of the proof is essentially analogous to that used in the previously mentioned papers where theorems similar to Theorem~\ref{thm:main} were established. However, in the current paper, we confront a number of delicate issues which are primarily imposed by the technical nature of ad\`elic Hilbert modular forms and the infinitude of the group of units in a totally real number field. In fact, these difficulties are most vividly portrayed in the treatment of the error term in equation (\ref{eqn:partial-error}). 
Dealing with the error term in the weight aspect case requires a more careful examination than the level aspect case considered in \cite{2016TRANS} and \cite{luo2}. One cannot achieve the desired estimate for the error term by the standard application of the Stirling formula and bounds for the Kloosterman sums and the $J$-Bessel functions. As explained in \cite{ganguly-hoffstein-sengupta}, this is due to the appearance of the parameter $k$ in the index of the $J$-Bessel function in addition to the gamma factors originating from the functional equation of the Rankin-Selberg $L$-function. In order to overcome this problem, one resorts to a trick (attributed in \cite{ganguly-hoffstein-sengupta} to Goldfeld \cite{goldfeld}) which amounts to opening up the Kloosterman sums and the $J$-Bessel functions, and then extracting from the expression an additive twist of a certain $L$-function to which a functional equation is applied. This allows for a convenient realignment of the sums in a way that makes it possible to adequately estimate the error term. In the setting of totally real number fields, this process is rendered even more complicated by the existence of a sum over totally positive units arising from the application of a Petersson-type trace formula for Hilbert modular forms. The reader is referred to Section~\ref{sec:error} for the detailed exposition. 

\medskip

The organization of this paper is as follows. In Section 1, we fix some notation pertaining to the totally real number field $F$ over which our work is based. We also recall some background material about (ad\`elic) Hilbert cusp forms and briefly explain how they correspond to the classical ones. We end this section with the definition and some basic properties of the Rankin-Selberg convolution of a pair of Hilbert cusp forms under certain hypotheses. In Section 2, we give a proof of the main theorem. The key ingredient is to study the twisted first moment given in (\ref{eqn:twisted-first-moment}). The application of an approximate functional equation and a Petersson trace formula allows us to split this sum into diagonal terms $M_{\g,\p}(\bk)$ and off-diagonal terms $E_{\g, \p}(\bk)$. Lemma~\ref{lem:main} gives asymptotic estimates for $M_{\g,\p}(\bk)$ and $E_{\g,\p}(\bk)$ as $k$ tends to infinity. Finally, Section~\ref{sec:main} and Section~\ref{sec:error} provide the detailed proof of Lemma~\ref{lem:main}.

\section{Notations and Preliminaries}\label{sec:background}

\subsection{The Base Field} Throughout the paper, we take the base field to be a totally real number field $F$ of degree $n$ over $\mathbb{Q}$, and we denote its ring of integers by $\ringO_{F}$. The absolute norm of an ideal $\a$ in $\ringO_{F}$ is defined as $\norm(\a)= [\ringO_{F}:\a]$. In fact, the absolute norm defined as such can be extended by multiplicativity to the group, $I(F)$, of fractional ideals of $F$. The trace and the norm of an element $x$ in the field extension $F/\mathbb{Q}$ are denoted by $\trace(x)$ and $\norm(x)$, respectively. Notice that for a principal ideal $(x)=x\ringO_{F}$, we have $\norm\left((x)\right)=|\norm(x)|$. The different ideal of $F$ and its discriminant over $\mathbb{Q}$ are denoted by $\Dif_{F}$ and $d_{F}$, respectively. We also have the identity $\norm(\Dif_{F})=|d_{F}|$.

The real embeddings of $F$ are denoted by $\sigma_{j}:x\mapsto x_{j}:=\sigma_{j}(x)$ for $j=1,\dots,n$. Once and for all, we fix an order of the embeddings, say $\sigma:=(\sigma_1,\dots,\sigma_n)$, so that any element $x$ in $F$ can be identified with the $n$-tuple $(x_{1},\dots,x_{n})$ in $\mathbb{R}^{n}$. This tuple may be, again, denoted by $x$ when no confusion arises. We say $x$ is totally positive and write $x\gg0$ if $x_{j}>0$ for all $j$, and for any subset $X\subset F$, we put $X^{+}=\{x\in X: x\gg0\}$. 

We denote the narrow class group of $F$ by $Cl^{+}(F)$. This is defined as the quotient group $I(F)/P^+(F)$, where $P^+(F)$ is the group of principal ideals generated by totally positive elements in $F$. It is well-known that $Cl^{+}(F)$ is a finite group, and its cardinality is denoted by $h_{F}^{+}$. We let  $\{\a_{1},\a_{2},\dots,\a_{h_{F}^{+}}\}$ be a fixed choice of representatives of the narrow ideal classes in $Cl^{+}(F)$. We write $\a\sim\b$ when fractional ideals $\a$ and $\b$ belong to the same narrow ideal class, in which case we have $\a=\xi\b$ for some $\xi$ in $F^{+}$. The symbol $[\a\b^{-1}]$ is used to refer to this element $\xi$ which is only unique up to multiplication by totally positive units in $\ringO_{F}$.

Let $\mathbb{A}_{F}$ be the ring of ad\`eles of $F$, and let $F_v$ be the completion of $F$ at a place $v$ of $F$. For a non-archimedean place $v$, we denote by $\ringO_v$ the local ring of integers. Furthermore, we let $F_\infty=\prod_{v|\infty} F_v$, where the product is taken over all archimedean places of $F$. In what follows, we make the identifications 
$$F_{\infty}=\mathbb{R}^{n},\quad \GL_{2}^{+}(F_{\infty})=\GL_{2}^{+}(\mathbb{R})^{n},\quad\text{and}\quad  \SO_{2}(F_{\infty})=\SO_{2}(\mathbb{R})^{n}.$$
In particular, each $r\in \SO_{2}(F_{\infty})$ can be expressed as 
$$r({\theta})=\left(r(\theta_{1}),\dots,r(\theta_n)\right)=\left(\left[ \begin{array}{cc}
\cos\theta_{j} & \sin\theta_{j}  \\ 
-\sin\theta_{j} & \cos\theta_{j} \end{array} \right]\right)_{j=1}^{n}.$$ 

Given an ideal $\n\subset\ringO_{F}$ and a non-archimedean place $v$ in $F$, we define the subgroup $K_{v}(\n)$ of $\GL_{2}(F_{v})$ as 
$$K_{v}(\n)=\left\{\left[ \begin{array}{cc}
a & b  \\
c & d \end{array} \right]\in \GL_{2}(\ringO_{v}): c\in \n \ringO_v\right\},$$ 
and put
\begin{equation*}\label{eqn:K_0}
K_{0}(\n)=\prod_{v<\infty}K_{v}(\n).
\end{equation*}

Before concluding the section, we recall the multi-index notation which is frequently used in this paper for convenience: For given $n$-tuples ${x}$ and ${z}$ and a scalar $a$, we set 
$${x}^{{z}}=\prod_{j=1}^{n}x_{j}^{z_{j}}\quad \text{and}\quad a^{{z}}=a^{\sum_{j=1}^{n}z_{j}}. $$
Such multi-index notation will also be employed to denote certain products of the gamma functions and the $J$-Bessel functions. See Sections~\ref{sec:RS_conv} and \ref{sec:moment}, respectively.

\subsection{Hilbert Modular Forms}\label{sec:HMF}
In this section, we recall the definition and some properties of the space of ad\`elic Hilbert modular forms, and we explain briefly the relation it bears to the space of classical Hilbert modular forms. Our exposition in the ad\`elic setting borrows heavily from that of Trotabas \cite[Section 3]{trotabas}. 

A complex-valued function $\f$ on $\GL_{2}(\mathbb{A}_F)$ is said to be an (ad\`elic) Hilbert cusp form of weight $\bk\in 2\mathbb{N}^{n}$, level $\n$, and with the trivial character if it satisfies the following conditions (Trotabas \cite[Definition~3.1]{trotabas}):
\begin{enumerate}
\item The identity $\f(\gamma xgr({\theta})u)=\f(g)\exp(i\bk{\theta})$ holds for all $\gamma\in\GL_{2}(F)$, $x\in\mathbb{A}_{F}^{\times}$, $g\in\GL_{2}(\mathbb{A}_{F})$, $r(\theta)\in\SO_{2}(F_{\infty})$, and $u\in K_{0}(\n)$.
\item Viewed as a smooth function on $\GL_{2}^{+}(F_{\infty})$, $\f$ is an eigenfunction of the Casimir element $\Delta:=(\Delta_{1},\dots,\Delta_{n})$ with eigenvalue $\displaystyle{\frac{k}{2}\left(1-\frac{k}{2}\right)}$.
\item We have $\displaystyle{\int_{F\backslash\mathbb{A}_{F}}\f\left(\left[ \begin{array}{cc}
1 & x  \\
0 & 1 \end{array} \right]g\right)dx=0}$ for all $g\in \GL_{2}(\mathbb{A}_{F})$ \hspace{.2in} (cuspidality condition).
\end{enumerate} 
We denote by $S_{\bk}(\n)$ the space of Hilbert cusp forms of weight $k$, level $\n$, and with the trivial character. 
\begin{remark}
As this paper only concerns forms with the trivial character, the specification of character will be omitted from now on. Moreover, we always take the weight vector $k$ to be in $2\mathbb{N}^{n}$, for otherwise the space $S_{\bk}(\n)$ would be trivial.
\end{remark}

Next, we introduce some notation which is needed to define the Fourier coefficients of a Hilbert cusp form $\f$. By the Iwasawa decomposition, any element $g$ in $\GL_{2}^{+}(F_{\infty})$ can be uniquely expressed as 
$$g=\left[ \begin{array}{cc}
{z} & 0  \\
0 & {z} \end{array} \right]\left[ \begin{array}{cc}
1 & {x}  \\
0 & 1 \end{array} \right]\left[ \begin{array}{cc}
{y} & 0  \\
0 & 1 \end{array} \right]r({\theta}),$$ 
with ${z}, {y}\in F_{\infty}^{+}$, ${x}\in F_{\infty}$, and $r({\theta})\in \SO_{2}(F_{\infty})$. Using the decomposition, we define $W_{\infty}^{0}(g)$ by
$$W_{\infty}^{0}(g)={y}^{\frac{k}{2}}\exp\left(2\pi i ({x}+i{y})\right)\exp\left(i\bk{\theta}\right).$$ In fact, the function $W_{\infty}^{0}$ is the new vector in the Whittaker model of the discrete series representation $\bigotimes_{j}\mathcal{D}(k_{j}-1)$ of $\GL_{2}(F_{\infty})$ (restricted to $\GL_{2}^{+}(F_{\infty})$). 

For $\f\in S_{\bk}(\n)$, $ g\in \GL_{2}^{+}(F_{\infty})$, and $\a\in I(F)$, we have the Fourier expansion 
\begin{equation*}\label{eqn:adelic-fourier-expansion}\f\left(g\left[ \begin{array}{cc}
\mathrm{id}(\a\Dif_F^{-1}) & 0  \\
0 & 1 \end{array} \right] \right)=\sum_{\nu\in(\a^{-1})^{+}}\frac{C(\nu,\a\Dif_F^{-1},\f)}{\norm(\nu\a)^{\frac{1}{2}}}W_{\infty}^{0}\left(\left[\begin{array}{cc}
\mathbf{\nu} & 0  \\
0 & 1 \end{array} \right] g\right),\end{equation*}
where $\mathrm{id}(\a\Dif_F^{-1})$ is the idele of $F$ associated with the ideal $\a\Dif_F^{-1}$. 
The Fourier coefficient of $\f$ at any integral ideal $\m$ in $\ringO_F$ is then defined as $C_\f(\m)=C(\nu,\a\Dif_F^{-1},\f)$, where $\a$ is a unique choice of representative in $\{\a_1,\dots,\a_{h_F^+}\}$ such that $\m\sim\a$, and $\nu=[\m\a^{-1}]$. Notice that $\nu$ is necessarily an element in $(\a^{-1})^+$, unique up to multiplication by $\ringO_F^{\times +}$. We say $\f$ is normalized if $C_{\f}(\ringO_{F})=1$.

We now retrieve the classical setting as it plays an important role in Section~\ref{sec:error}. Indeed, it is well-known to experts that a form $\f$ in $S_{\bk}(\n)$ can be viewed as an $h_{F}^{+}$-tuple of classical Hilbert cusp forms $f_{\a_i}$ indexed by the narrow class representatives $\{\a_1,\dots,\a_{h_{F}^{+}}\}$. Each function $f_{\a_{i}}$ on $\mathbb{H}^{n}$ is defined via the map
\begin{equation}\label{eqn:classicalHMF}
 z:=x+iy \mapsto f_{\a_{i}}(z)=y^{-\frac{k}{2}}\f\left(\left[ \begin{array}{cc} y & x  \\ 0 & 1 \end{array} \right]\left[ \begin{array}{cc} \mathrm{id}(\a_i\Dif_{F}^{-1}) & 0  \\ 0 & 1 \end{array} \right] \right),
 \end{equation}
which yields a classical Hilbert cusp form of weight $k$ with respect to the congruence subgroup
\[\Gamma_0(\n,\a_i\Dif_F^{-1})=\left\{\left[\begin{array}{cc} a&b\\c&d\end{array}\right]\in\GL_2^+(F)\,:\, a, d\in\ringO_F, b\in\a_i\Dif_F^{-1},c\in\n\a_i^{-1}\Dif_F,ad-bc\in\ringO_F^{\times +}\right\}.\]
This means that each $f_{\a_i}$ satisfies the automorphy condition, $f_{\a_{i}}|\!|_k\gamma=f_{\a_i}$, for all $\gamma$ in $\Gamma_0(\n,\a_{i}\Dif_F^{-1})$. Here we recall that $|\!|_k$ is the weight $k$ slash operator given by
\[ f|\!|_k\gamma(z)=\left(\det \gamma\right)^{k/2}j(\gamma, z)^{-k}f(\gamma z).\]
It follows from equation (\ref{eqn:classicalHMF}) that the Fourier coefficients of 
\begin{equation*}\label{eqn:classical-fourier-expansion}f_{\a_i}(z)=\sum_{\nu\in(\a_i^{-1})^+}a_\nu(f_{\a_i})\exp(2\pi i \trace(\nu z))\end{equation*}
 are given by 
\begin{equation}\label{eqn:classical-fourier-coefficients} a_\nu(f_{\a_i})=\frac{C(\nu,\a_i\Dif_F^{-1}, \f)}{\norm(\nu\a_i)^{1/2}}\nu^{\frac{k}{2}}=\frac{C_\f(\m)}{\norm(\m)^{1/2}}\nu^{\frac{k}{2}}.\end{equation}
The reader is referred to, for instance, Garrett \cite[Chapter 1, 2]{garrett}, Raghuram-Tanabe \cite[Section 4]{JRMS-2011}, and Shimura \cite[Section 2]{shimura-duke} for more details on this realization.

The space of cusp forms can be decomposed as $S_k(\n)=S_k^{\mathrm{old}}(\n)\oplus S_k^{\mathrm{new}}(\n)$ where $S_k^{\mathrm{old}}(\n)$ is the subspace of cusp forms that come from lower levels. The new space $S_k^{\mathrm{new}}(\n)$ is the orthogonal complement of $S_{\bk}^{\mathrm{old}}(\n)$ with respect to the inner product 
$$\left<\f,\bf{h}\right>_{S_{\bk}(\n)}=\int_{\GL_{2}(F)\mathbb{A}_{F}^{\times}\backslash \GL_{2}(\mathbb{A}_{F})/K_{0}(\n)}\f(g)\overline{{\bf{h}}(g)}\;dg.$$

The space $S_k(\n)$ has an action of Hecke operators $\{T_{\m}\}_{\m\subset\ringO_{F}}$ much like the classical setting over $\mathbb{Q}$. A Hilbert cusp form $\f$ in $S_k(\n)$ is said to be primitive if it is a normalized common Hecke eigenfunction in the new space. We denote by $\Pi_{\bk}(\n)$ the (finite) set of all primitive forms of weight $\bk$ and level $\n$. If $\f$ is a function in $\Pi_{\bk}(\n)$, it follows from the work of Shimura \cite{shimura-duke} that the coefficients $C_{\f}(\m)$ are equal to the Hecke eigenvalues for $T_{\m}$ for all $\m$. Moreover, since $\f$ is with the trivial character, the coefficients $C_{\f}(\m)$ are known to be real for all $\m$.

\subsection{Rankin-Selberg Convolutions}\label{sec:RS_conv}
We now recall the construction of Rankin-Selberg convolutions of two Hilbert modular forms following Shimura~\cite[Section~4]{shimura-duke}. We note, however, that our normalization is slightly different from what Shimura uses. Consider two primitive forms $\f\in \Pi_{\bk}(\q)$ and $\g\in \Pi_{\bl}(\n)$, where we assume that $\q$ and $\n$ are coprime. The Rankin-Selberg $L$-function associated with $\f$ and $\g$ is defined as 
$$L(\f\otimes \g,s)=\zeta_{F}^{\n\q}(2s)\sum_{\m\subset \ringO_{F}}\frac{C_{\f}(\m)C_{\g}(\m)}{\norm(\m)^{s}},$$ 
where $\zeta_F^{\n\q}(s)$ is the Dedekind zeta function of $F$ away from $\n\q$, that is,
 $$\zeta_{F}^{\n\q}(2s)=\zeta_{F}(2s)\prod_{\substack{\l|\n\q \\ \l\,:\, \text{prime}}}(1-\norm(\l)^{-2s}).$$ 
This series is absolutely convergent for $\Re(s)>1$ since the Fourier coefficients of $\f$ and $\g$ satisfy the Ramanujan-Petersson bound (proven by Blasius in \cite{blasius}):
\begin{equation*}\label{eqn:ramanujan-bound}
C_{\f}(\m)\ll_{\epsilon}\norm(\m)^{\epsilon}\quad\text{and}\quad C_{\g}(\m)\ll_{\epsilon}\norm(\m)^{\epsilon}.
\end{equation*}
It is useful in applications to write the series $L(\f\otimes \g,s)$ as a sum over all positive integers. In fact, it is easy to see that  
$$L(\f\otimes \g,s)=\sum_{m=1}^\infty\frac{b_{m}^{\n\q}(\f\otimes\g)}{m^s},$$
where
$$ b_{m}^{\n\q}(\f\otimes\g)=\sum_{d^2|m}\left(a_d(\n\q)\sum_{\norm(\m)=m/d^2}C_\f(\m)C_\g(\m)\right)$$
with $a_d(\n\q)$ being the number of integral ideals of norm $d$ that are coprime to $\n\q$. 

Let the archimedean part of this $L$-function be 
 $$L_{\infty}(\f\otimes\g,s)=\Gamma\left(s+\frac{|\bk-{\bl}|}{2}\right)\Gamma\left(s-1+\frac{\bk+\bl}{2}\right),$$
  with
$$\Gamma\left(s+\frac{|\bk-{\bl}|}{2}\right)=\prod_{j=1}^{n}\Gamma\left(s+\frac{|k_j-l_j|}{2}\right)\;\;\text{and}\;\;\Gamma\left(s-1+\frac{\bk+\bl}{2}\right)=\prod_{j=1}^{n}\Gamma\left(s-1+\frac{k_j+l_j}{2}\right).$$
Define the completed $L$-function 
$$\Lambda(\f\otimes \g,s)=(2\pi)^{-2sn-\kappa}\norm(\n\q\Dif_F^2)^{s}L_{\infty}(\f\otimes\g,s)L(\f\otimes \g,s),$$
with $\kappa:=(\kappa_1,\dots,\kappa_n)$ and $\kappa_j=\mathrm{max}\{k_j,l_j\}$.
Then $\Lambda(\f\otimes \g,s)$ has analytic continuation to $\mathbb{C}$ as an entire function and satisfies the functional equation \begin{equation}\label{functional-equation}\Lambda(\f\otimes\g,s)=\Lambda(\f\otimes \g,1-s).\end{equation}
See, for instance, Prasad-Ramakrishnan \cite{prasad-ramakrishnan} for the details.

We note that we fix the level of $\f$ to be $\q=\ringO_F$ from now on. Also, since our aim is to study the asymptotic behavior (as $k_j$ tends to infinity) of an average expression over $\f$ as it appears in Theorem~\ref{thm:main}, one can freely assume that $k_{j}>l_{j}$ for all $j=1,\dots,n$. This condition will be imposed for the rest of this paper. We remark that the main result of this paper will hold, with minor adjustments to the proof, when only $\max_{j}\{k_j\}$ goes to infinity.

\section{Proof of Theorem~\ref{thm:main}}

\subsection{Twisted First Moment}\label{sec:moment}

Let $\g$ be a fixed form in $\Pi_{\bl}(\n)$, and let $\p$ be an ideal in $F$ which is either $\ringO_{F}$ or prime. We consider the twisted first moment 
\begin{equation}\label{eqn:twisted-first-moment} \sum_{\f\in\Pi_{\bk}(\ringO_F)}L\left(\f\otimes \g,\frac12\right)C_\f(\p)\omega_\f,
\end{equation} 
where
\[ \omega_\f=\frac{\Gamma({k-1})}{(4\pi)^{\bk-1}|d_F|^{1/2}\left<\f,\f\right>_{S_{k}(\ringO_F)}}.\]  
The primary goal of this paper is to establish an asymptotic formula of this moment as $k$ approaches infinity. Indeed, our main theorem is a direct application of such a formula as will become evident in the next section (Section~\ref{sec:proof}). In order to analyze the sum in equation (\ref{eqn:twisted-first-moment}), we utilize two standard results from analytic number theory tailored to the setting of Hilbert modular forms; namely, an approximate functional equation and a Petersson trace formula which we now proceed to describe.

\begin{prop}[Approximate functional equation]\label{approximate-functional-equation}
Let $G(u)$ be a holomorphic function on an open set containing the strip $|\Re(u)|\leq 3/2$ and bounded therein, satisfying $G(u)=G(-u)$, and $G(0)=1$. Then we have 
\[ L\left(\f\otimes \g,\frac12\right)=2\sum_{m=1}^\infty\frac{b_{m}^{\n}(\f\otimes\g)}{\sqrt{m}}V_{1/2}\left(\frac{4^{n}\pi^{2n}m}{\norm(\n\Dif_F^2)}\right),\]
where 
\begin{equation}\label{eqn:v}
V_{s}(y)=\frac{1}{2\pi i}\int_{(3/2)}y^{-u}\gamma(s, u)G(u)\frac{du}{u}
\end{equation}
with
\[ \gamma(s,u)=\frac{L_\infty(\f\otimes \g,s+u)}{L_\infty( \f\otimes \g,s)}.\]
Moreover, the derivatives of $V_{1/2}(y)$ satisfy 
$$y^{a}V_{1/2}^{(a)}(y)\ll\left(1+\frac{y}{k^{2}}\right)^{-A} \quad\text{and}\quad
y^{a}V_{1/2}^{(a)}(y)=\delta_{a}+O\left(\left(\frac{y}{k^{2}}\right)^{\alpha}\right)$$ for some $0<\alpha\leq1$. The value of $\delta_a$ is taken to be $1$ if $a=0$ and $0$ otherwise, 
and the implied constants depend on $a$, $A$, and $\alpha$.
\end{prop}

\begin{proof}
See the authors' level aspect paper \cite[Section~4]{2016TRANS}. The estimates on $V_{1/2}^{(a)}(y)$ follow from Iwaniec-Kowalski \cite[Proposition~5.4]{iwaniec-kowalski}.
\end{proof}

The function $G(u)$ appearing in Proposition~\ref{approximate-functional-equation} is set to be $e^{u^2}$ throughout this paper. Another crucial tool to our work is a Petersson trace formula due to Trotabas (see \cite[Proposition~6.3]{trotabas}) which states the following.

\begin{prop}[Petersson trace formula]\label{trace-formula}
Let $\q$ be an integral ideal in $F$. Let $\a$ and $\b$ be fractional ideals in $F$. For $\alpha\in\a^{-1}$ and $\beta\in\b^{-1}$, we have \begin{align*} &\sum_{\f\in H_{\bk}(\q)}\frac{\Gamma(\bk-\boldsymbol{1})}{(4\pi)^{\bk-\boldsymbol{1}}|d_{F}|^{1/2}\left<\f,\f\right>_{S_{\bk}(\q)}}C_{\f}(\alpha\a)\overline{C_{\f}(\beta\b)}\\
&\hspace{.6in}=\1_{\alpha\a=\beta\b}+C\sum_{\substack{\mathfrak{c}^{2}\sim\a\b\\c\in\c^{-1}\q\backslash\{0\}\\\epsilon\in\mathcal{O}_{F}^{\times+}/\mathcal{O}_{F}^{\times2}}}\frac{{\mathit Kl}(\epsilon\alpha,\a; \beta,\b;c,\c)}{\norm(c\c)}J_{\bk-1}\left(\frac{4\pi\sqrt{\epsilon\alpha\beta\left[\mathfrak{a}\mathfrak{b}\mathfrak{c}^{-2}\right]}}{|c|}\right),\end{align*}
where $\dis{C=\frac{(-1)^{\bk/2}(2\pi)^{n}}{2|d_{F}|^{1/2}}}$, and $H_{\bk}(\q)$ is an orthogonal basis for the space $S_{\bk}(\q)$.
\end{prop}

In the proposition above, multi-index notation is used for the $J$-Bessel function. Indeed, we have \[J_{\bk-1}\left(\frac{4\pi\sqrt{\epsilon\alpha\beta\left[\mathfrak{a}\mathfrak{b}\mathfrak{c}^{-2}\right]}}{|c|}\right)=\prod_{j=1}^{n}J_{k_{j}-1}\left(\frac{4\pi\sqrt{\epsilon_{j}\alpha_{j}\beta_{j}\left[\mathfrak{a}\mathfrak{b}\mathfrak{c}^{-2}\right]_{j}}}{|c_{j}|}\right).\] For each $j$, the corresponding factor in the above product is an evaluation of the classical $J$-Bessel function $J_{k_{j}-1}$ which could be written via the Mellin-Barnes integral representation as
\begin{equation}\label{eqn:j-bessel}
J_{k_{j}-1}(x)=\int_{(\sigma)}\frac{\Gamma\left(\frac{k_{j}-1-s}{2}\right)}{\Gamma\left(\frac{k_{j}-1+s}{2}+1\right)}\left(\frac{x}{2}\right)^{s}\;ds\quad\quad \text{with}\;\;0<\sigma<k_{j}-1.
\end{equation}

 As for the Kloosterman sum, it is defined as follows. Given two fractional ideals $\a$ and $\b$, let $\c$ be an ideal such that $\c^{2}\sim\a\b$. For $\alpha\in\a^{-1}$, $\beta\in\b^{-1}$, and $c\in\c^{-1}$, the Kloosterman sum $\mathit{Kl}(\alpha,\a;\beta,\b;c,\c)$ is given by 

\begin{equation*}\label{eqn:kloosterman}
{\mathit Kl}(\alpha,\a; \beta,\b;c,\c)=\sum_{x\in \left(\a\Dif_{F}^{-1}\c^{-1}/\a\Dif_{F}^{-1}c\right)^{\times}}\exp\left(2\pi i\trace\left(\frac{\alpha x+\beta\left[\a\b\c^{-2}\right]\overline{x}}{c}\right)\right).
\end{equation*}
Here $\overline{x}$ is the unique element in $\left(\a^{-1}\Dif_{F}\c/\a^{-1}\Dif_{F}c\c^{2}\right)^{\times}$ such that $x\overline{x}\equiv 1\mod c\c$. The reader is referred to Section 2.2 and Section 6 in \cite{trotabas} for more details on this construction. 

We are now ready to start our investigation of the sum in (\ref{eqn:twisted-first-moment}). We begin by applying the approximate functional equation provided in Proposition \ref{approximate-functional-equation}. Then, a simple re-arrangement of the sums yields
\begin{align*} &\sum_{\f\in\Pi_{\bk}(\ringO_F)}L\left(\f\otimes \g,\frac12\right)C_\f(\p)\omega_\f \\
&= \sum_{\f\in\Pi_{\bk}(\ringO_F)} 2\sum_{m=1}^\infty\frac{b_{m}^{\n}(\f\otimes\g)}{\sqrt{m}}V_{1/2}\left(\frac{4^{n}\pi^{2n}m}{\norm(\n\Dif_F^2)}\right)C_\f(\p)\omega_\f\\
&= 2\sum_{m=1}^\infty \frac{1}{\sqrt{m}}V_{1/2}\left(\frac{4^{n}\pi^{2n}m}{\norm(\n\Dif_F^2)}\right)\sum_{\f\in\Pi_{\bk}(\ringO_F)}\omega_\f C_\f(\p)\sum_{d^2|m}a_d(\n)\sum_{\norm(\m)=m/d^2}C_\f(\m)C_\g(\m)\\
&= 2\sum_{m=1}^\infty \frac{1}{\sqrt{m}}V_{1/2}\left(\frac{4^{n}\pi^{2n}m}{\norm(\n\Dif_F^2)}\right)\sum_{d^2|m}a_d(\n)\sum_{\norm(\m)=m/d^2}C_\g(\m)\sum_{\f\in\Pi_{\bk}(\ringO_F)}\omega_\f C_\f(\p)C_\f(\m)\\
&=2\sum_{\m\subset\ringO_{F}} \frac{C_{\g}(\m)}{\sqrt{\norm(\m)}}\sum_{d=1}^{\infty}\frac{a_d(\n)}{d}V_{1/2}\left(\frac{4^n\pi^{2n}\norm(\m)d^2}{\norm(\n\Dif_F^2)}\right)\sum_{\f\in\Pi_{\bk}(\ringO_F)}\omega_\f C_\f(\p)C_\f(\m).
\end{align*}

Recall that we denote by $\{\a_{i}\}$ a fixed system of representatives for $Cl^{+}(F)$. For an ideal $\m$ of $\ringO_{F}$, we write $\m=\nu\a$ for some $\a\in\{\a_{i}\}$ and $\nu\in(\a^{-1})^{+}\mod \ringO_F^{\times+}$ as seen earlier. Similarly, we write $\p$ as $\p=\xi\b$ with  $\b\in\{\a_{i}\}$ and $\xi\in(\b^{-1})^{+}\mod \ringO_F^{\times+}$. 
Hence, applying the Petersson trace formula to the above expression, we obtain
\begin{eqnarray*}
&& \sum_{\f\in\Pi_{\bk}(\ringO_F)}L\left(\f\otimes \g,\frac12\right)C_\f(\p)\omega_\f\\
&&=2\sum_{\{\a_{i}\}}\sum_{\nu\in(\a_{i}^{-1})^{+}/\ringO_{F}^{\times +}} \frac{C_{\g}(\nu\a_{i})}{\sqrt{\norm(\nu\a_{i})}}\sum_{d=1}^{\infty}\frac{a_d(\n)}{d}V_{1/2}\left(\frac{4^n\pi^{2n}\norm(\nu\a_{i})d^2}{\norm(\n\Dif_F^2)}\right)\\
&&\hspace{.2in}\times\left\{\1_{\nu\a_{i}=\xi\b}+C\sum_{\substack{\c^{2}\sim\a_{i}\b\\c\in\c^{-1}\backslash\{0\}\\\epsilon\in\ringO_{F}^{\times+}/\ringO_{F}^{\times2}}}\frac{{\mathit Kl}(\epsilon\nu,\a_{i}; \xi,\b;c,\c)}{\norm(c\c)}J_{k-1}\left(\frac{4\pi\sqrt{\epsilon\nu\xi\left[\a_{i}\b\c^{-2}\right]}}{|c|}\right)\right\}.
\end{eqnarray*} 
For convenience, we write
\begin{equation}\label{eqn:main+error+old}
\sum_{\f\in\Pi_{\bk}(\ringO_F)}L\left(\f\otimes \g,\frac12\right)C_\f(\p)\omega_\f=M_{\g,\p}(k)+E_{\g,\p}(k),\end{equation}
where 
\begin{equation}\label{eqn:main}
M_{\g,\p}(k)=2\frac{C_{\g}(\p)}{\sqrt{\norm(\p)}}\sum_{d=1}^{\infty}\frac{a_d(\n)}{d}V_{1/2}\left(\frac{4^n\pi^{2n}\norm(\p)d^2}{\norm(\n\Dif_F^2)}\right),
\end{equation}
and
\begin{align}\label{eqn:error}
E_{\g,\p}(k)&=2C\sum_{\{\a_{i}\}}\sum_{\nu\in(\a_{i}^{-1})^{+}/\ringO_{F}^{\times +}} \frac{C_{\g}(\nu\a_{i})}{\sqrt{\norm(\nu\a_{i})}}\sum_{d=1}^{\infty}\frac{a_d(\n)}{d}V_{1/2}\left(\frac{4^n\pi^{2n}\norm(\nu\a_{i})d^2}{\norm(\n\Dif_F^2)}\right)\\
&\hspace{.7in}\times\sum_{\substack{\c^{2}\sim\a_{i}\b\\c\in\c^{-1}\backslash\{0\}\\ \epsilon\in\ringO_{F}^{\times+}/\ringO_{F}^{\times2}}}\frac{{\mathit Kl}(\epsilon\nu,\a_{i}; \xi,\b;c,\c)}{\norm(c\c)}J_{k-1}\left(\frac{4\pi\sqrt{\epsilon\nu\xi\left[\a_{i}\b\c^{-2}\right]}}{|c|}\right). \nonumber
\end{align}

In the following section, we introduce asymptotic estimates on (\ref{eqn:main}) and (\ref{eqn:error}), which we use to complete the proof of Theorem~\ref{thm:main}.

\subsection{Proof of Main Theorem}\label{sec:proof}

As mentioned at the beginning of Section~\ref{sec:moment}, the proof of our main theorem hinges upon an asymptotic formula for the twisted first moment (\ref{eqn:twisted-first-moment}). In fact, the desired formula follows immediately from the lemma below.

\begin{lemma}\label{lem:main}
Let $M_{\g,\p}(k)$ and $E_{\g,\p}(k)$ be as in (\ref{eqn:main}) and (\ref{eqn:error}). As $k$ approaches infinity, we have the following estimates:
\begin{enumerate}
\item $\dis M_{\g,\p}(k)=2\frac{C_\g(\p)}{\sqrt{\norm(\p)}} \gamma_{-1}^{\n}(F)\log(k)+O(1)$,\vskip .1in
 where $\gamma_{-1}^{\n}(F)$ is twice the residue of $\zeta_{F}^{\n}(2u+1)$ at $u=0$, and $\log (k)=\sum_{j=1}^{n}\log(k_{j})$.
\item $\dis E_{\g,\p}(k)=O(1)$.
\end{enumerate}
\end{lemma}

\begin{proof}
Sections ~\ref{sec:main} and \ref{sec:error} are devoted to proving the first and second statements, respectively. 
\end{proof}

Applying the lemma above to equation (\ref{eqn:main+error+old}) yields the following proposition.
\begin{prop}\label{prop:moment-asymptotic-formula}
Let $\g\in \Pi_{l}(\n)$, and let $\p$ be either $\ringO_{F}$ or a prime ideal. Then, we have \[\sum_{\f\in\Pi_{\bk}(\ringO_F)}L\left(\f\otimes \g,\frac12\right)C_\f(\p)\omega_\f=2\frac{C_\g(\p)}{\sqrt{\norm(\p)}} \gamma_{-1}^{\n}(F)\log(k)+O(1) \quad (k\to \infty).\]
\end{prop}

Finally, we complete the proof of Theorem~\ref{thm:main}: Let $\g$ and $\g'$ be primitive forms satisfying all the hypotheses given in Theorem~\ref{thm:main}. 
Applying Proposition~\ref{prop:moment-asymptotic-formula} with $\p=\ringO_F$ gives $\gamma_{-1}^\n(F)=\gamma_{-1}^{\n'}(F)$. A second application of the proposition, with any prime $\p$ not dividing $\n\n'$, will then imply that $C_{\g}(\p)=C_{\g'}(\p)$.  It follows that the Hecke eigenvalues of $\g$ and $\g'$ for $T_\p$ are equal. Therefore, we have $\g=\g'$ by the strong multiplicity one theorem (cf. Bump \cite[Chapter~3]{bump} and Miyake \cite{miyake}). 

\section{Contribution of $M_{\g,\p}(k)$}\label{sec:main}
In this section, we establish an asymptotic formula in the weight aspect for the diagonal terms 
$$M_{\g,\p}(k)=2\frac{C_{\g}(\p)}{\sqrt{\norm(\p)}}\sum_{d=1}^{\infty}\frac{a_d(\n)}{d}V_{1/2}\left(\frac{4^n\pi^{2n}\norm(\p)d^2}{\norm(\n\Dif_F^2)}\right).$$ 
The summation over $d\in\Z_{>0}$ in the above expression can be evaluated as follows. Since we have
\begin{align*}
\sum_{d=1}^{\infty}\frac{a_d(\n)}{d}V_{1/2}\left(\frac{4^n\pi^{2n}\norm(\p)d^2}{\norm(\n\Dif_F^2)}\right)&=\frac{1}{2\pi i}\sum_{d=1}^{\infty}\frac{a_d(\n)}{d}\int_{(3/2)}\left(\frac{4^n\pi^{2n}\norm(\p)d^2}{\norm(\n\Dif_F^2)}\right)^{-u} \gamma\left(\frac12,u\right)G(u)\;\frac{du}{u}\\
&=\frac{1}{2\pi i}\int_{(3/2)}\left(\frac{4^n\pi^{2n}\norm(\p)}{\norm(\n\Dif_F^2)}\right)^{-u} \gamma\left(\frac12,u\right) G(u)\zeta_{F}^{\n}(2u+1)\;\frac{du}{u},
\end{align*}
where $\gamma(1/2,u)$ is defined as in Proposition~\ref{approximate-functional-equation},
shifting the contour of integration to $\Re(u)=-1/2$ gives 
\begin{align}\label{eqn:main_res}
\sum_{d=1}^{\infty}\frac{a_d(\n)}{d}V_{1/2}\left(\frac{4^n\pi^{2n}\norm(\p)d^2}{\norm(\n\Dif_F^2)}\right)&=\mathop{Res}_{u=0}\left(\left(\frac{4^n\pi^{2n}\norm(\p)}{\norm(\n\Dif_F^2)}\right)^{-u}\gamma\left(\frac12,u\right)G(u)\frac{\zeta_{F}^{\n}(2u+1)}{u}\right)\\
&\hspace{.4in}+\frac{1}{2\pi i}\int_{(-1/2)}\left(\frac{4^n\pi^{2n}\norm(\p)}{\norm(\n\Dif_F^2)}\right)^{-u}\gamma\left(\frac12, u\right)G(u)\zeta_{F}^{\n}(2u+1)\;\frac{du}{u}. \nonumber
\end{align}
We write the above integral as
$$I=-\frac{1}{2\pi}\int_{-\infty}^\infty \left(\frac{4^n\pi^{2n}\norm(\p) }{\norm(\n\Dif_F^2)}\right)^{\frac12-it}\gamma\left(\frac12,-\frac12+it\right)G\left(-\frac12+it\right)\zeta_F^{\n}(2it)\frac{dt}{-\frac12+it}.$$ 
In order to obtain an estimate for the $\gamma$-factor in the integral $I$, we recall the following bound on gamma quotients.

\begin{lemma}{\rm (\cite[Lemma 1]{ganguly-hoffstein-sengupta})}\label{gamma-quotient}
Suppose $A>0$ and $c$ is a real constant such that $\dis |c|<A/2$. Then
\[\frac{\Gamma(A+c+it)}{\Gamma(A+it)}\ll|A+it|^c.\]
\end{lemma}
Observe that Lemma \ref{gamma-quotient}, along with the trivial bound $|\Gamma(x+it)|\leq|\Gamma(x)|$, yields
\begin{align*}
\gamma\left(\frac12,-\frac12+it\right)&=\frac{\Gamma\left(-\frac12+it+\frac{k-l+1}{2}\right)}{\Gamma\left(it+\frac{k-l+1}{2}\right)}\frac{\Gamma\left(it+\frac{k-l+1}{2}\right)}{\Gamma\left(\frac{k-l+1}{2}\right)}\frac{\Gamma\left(-\frac12+it+\frac{k+l-1}{2}\right)}{\Gamma\left(it+\frac{k+l-1}{2}\right)}\frac{\Gamma\left(it+\frac{k+l-1}{2}\right)}{\Gamma\left(\frac{k+l-1}{2}\right)}\\
&\ll\left|\frac{k-l+1}{2}\right|^{-1/2}\left|\frac{k+l-1}{2}\right|^{-1/2}\ll{k}^{-1}.\end{align*}
Here, we remind the reader that $k^{-1}$ is, in fact,  $\prod_{j=1}^n k_j^{-1}$. By setting $G(u)=e^{u^2}$, we conclude that the integral $I$ satisfies
\[ I \ll k^{-1}\int_{-\infty}^\infty \left|\frac{e^{-t^2+\frac14}}{-\frac12+it}\right|\cdot|\zeta_F^{\n}(2it)|dt\ll  k^{-1}.\]
To compute the residue at $u=0$ in (\ref{eqn:main_res}), we recall that
\begin{eqnarray*}
e^{u^2}&=&1+u^2+\cdots,\\
\frac{\Gamma(a+u)}{\Gamma(a)}&=&1+\frac{\Gamma^\prime(a)}{\Gamma(a)}u+\cdots,\\
\left(\frac{4^n\pi^{2n}\norm(\p)}{\norm(\n\Dif_F^2)}\right)^{-u}&=&1-\log\left(\frac{4^n\pi^{2n}\norm(\p)}{\norm(\n\Dif_F^2)}\right)u+\cdots, \\
\zeta_F^{\n}(2u+1)&=&\frac{\gamma_{-1}^{\n}(F)}{2u}+\gamma_0^{\n}(F)+\cdots.\end{eqnarray*}
Using the above Taylor expansions, we obtain
\begin{align*} &\mathop{Res}_{u=0}\left(\left(\frac{4^n\pi^{2n}\norm(\p)}{\norm(\n\Dif_F^2)}\right)^{-u}\gamma\left(\frac12,u\right)\frac{e^{u^2}\zeta_{F}^{\n}(2u+1)}{u}\right)\\
&\hspace{.7em}=\gamma_0^{\n}(F)+\frac{\gamma_{-1}^{\n}(F)}{2}\left(\sum_{j=1}^n\frac{\Gamma^\prime}{\Gamma}\left(\frac{k_j-l_j+1}{2}\right)+\frac{\Gamma^\prime}{\Gamma}\left(\frac{k_j+l_j-1}{2}\right)-\log\left(\frac{4^n\pi^{2n}\norm(\p)}{\norm(\n\Dif_F^2)}\right)\right).\end{align*}
Moreover, we know by Stirling's formula that $$\frac{\Gamma'}{\Gamma}(a)=\log a-\frac{1}{2a}+O\left(\frac{1}{|a|^2}\right),$$
and therefore we conclude that $$M_{\g,\p}(k)=2\frac{C_{\g}(\p)}{\sqrt{\norm(\p)}}\gamma_{-1}^{\n}(F)\log\left(k\right)+O(1)\quad\quad k\to\infty.$$ 
This proves the first statement of Lemma~\ref{lem:main}.

\section{Contribution of $E_{\g,\p}(k)$}\label{sec:error}
In order to give an asymptotic estimate for the off-diagonal terms $E_{\g,\p}(k)$, it suffices to fix a representative $\a$ in $\{\a_i\}$ and consider a partial sum $E_{\g,\p,\a}(k)$ given by 
\begin{align}\label{eqn:partial-error}
E_{\g,\p,\a}(k)&=\sum_{\nu\in(\a^{-1})^{+}/ \ringO_{F}^{\times +}} \frac{C_{\g}(\nu\a)}{\sqrt{\norm(\nu\a)}}\sum_{d=1}^{\infty}\frac{a_d(\n)}{d}V_{1/2}\left(\frac{4^n\pi^{2n}\norm(\nu\a)d^2}{\norm(\n\Dif_F^2)}\right)\\
&\hspace{.5in} \times \sum_{c\in\c^{-1}\backslash\{0\}/\ringO_{F}^{\times +}}\sum_{\substack{\epsilon\in\ringO_{F}^{\times+}/\ringO_{F}^{\times2} \\ \eta\in\ringO_{F}^{\times+}}}\frac{{\mathit Kl}(\epsilon\nu,\a; \xi,\b;c\eta,\c)}{\norm(c\c)}J_{k-1}\left(\frac{4\pi\sqrt{\epsilon\nu\xi\left[\a\b\c^{-2}\right]}}{\eta|c|}\right). \nonumber
\end{align} 
Notice here that we are also fixing a narrow class representative $\c$ such that $\c^{2}\sim\a\b$.
Interchanging the summations and applying some algebraic manipulations, we see that
\begin{eqnarray*}
E_{\g,\p,\a}(k)&=& \sum_{d=1}^{\infty}\frac{a_d(\n)}{d}\sum_{\nu\in(\a^{-1})^{+}/ \ringO_{F}^{\times +}} \frac{C_{\g}(\nu\a)}{\sqrt{\norm(\nu\a)}}V_{1/2}\left(\frac{4^n\pi^{2n}\norm(\nu\a)d^2}{\norm(\n\Dif_F^2)}\right)\\
&& \hspace{.5in} \times \sum_{c\in\c^{-1}\backslash\{0\}/\ringO_{F}^{\times +}}\sum_{\substack{\epsilon\in\ringO_{F}^{\times+}/\ringO_{F}^{\times2} \\ \eta\in\ringO_{F}^{\times+}}}\frac{{\mathit Kl}(\epsilon\eta^{-2}\nu,\a; \xi,\b;c,\c)}{\norm(c\c)}J_{k-1}\left(\frac{4\pi\sqrt{\epsilon\eta^{-2}\nu\xi\left[\a\b\c^{-2}\right]}}{|c|}\right).
\end{eqnarray*}
Replacing the sums over $\epsilon$ and $\eta$ by the sum over all totally positive units, we obtain 
\begin{eqnarray*}
E_{\g,\p,\a}(k)&=&\sum_{d=1}^{\infty}\frac{a_d(\n)}{d}\sum_{\nu\in(\a^{-1})^{+}/ \ringO_{F}^{\times +}} \frac{C_{\g}(\nu\a)}{\sqrt{\norm(\nu\a)}}V_{1/2}\left(\frac{4^n\pi^{2n}\norm(\nu\a)d^2}{\norm(\n\Dif_F^2)}\right)\\
&& \hspace{.5in} \times \sum_{c\in\c^{-1}\backslash\{0\}/\ringO_{F}^{\times +}}\sum_{\eta\in\ringO_{F}^{\times+}}\frac{{\mathit Kl}(\eta\nu,\a; \xi,\b;c,\c)}{\norm(c\c)}J_{k-1}\left(\frac{4\pi\sqrt{\eta\nu\xi\left[\a\b\c^{-2}\right]}}{|c|}\right) \\
&=&\sum_{d=1}^{\infty}\frac{a_d(\n)}{d}\sum_{\nu\in(\a^{-1})^{+}} \frac{C_{\g}(\nu\a)}{\sqrt{\norm(\nu\a)}}V_{1/2}\left(\frac{4^n\pi^{2n}\norm(\nu\a)d^2}{\norm(\n\Dif_F^2)}\right)\\
&& \hspace{.5in} \times \sum_{c\in\c^{-1}\backslash\{0\}/\ringO_{F}^{\times +}}\frac{{\mathit Kl}(\nu,\a; \xi,\b;c,\c)}{\norm(c\c)}J_{k-1}\left(\frac{4\pi\sqrt{\nu\xi\left[\a\b\c^{-2}\right]}}{|c|}\right),
\end{eqnarray*}
where the last equality follows from unfolding the sum over $(\a^{-1})^+/\ringO_F^{\times +}$ to a sum over $(\a^{-1})^+$. 
In what follows, we shall open up the Kloosterman sum ${\mathit Kl}(\nu,\a; \xi,\b;c,\c)$ while also replacing the terms $V_{1/2}( *)$ and $J_{k-1}(*)$ by their integral representations given in (\ref{eqn:v}) and (\ref{eqn:j-bessel}). 
More precisely, we write $E_{\g,\p,\a}(k)$ as
\begin{eqnarray*} 
E_{\g,\p,\a}(k)&=&\sum_{d=1}^{\infty}\frac{a_d(\n)}{d}\sum_{\nu\in(\a^{-1})^{+}}\frac{C_{\g}(\nu\a)}{\sqrt{\norm(\nu\a)}}\int_{(3/2)}\left(\frac{(2\pi)^{2n}\norm(\nu\a)d^2}{\norm(\n\Dif_F^2)}\right)^{-u}\gamma\left(\frac12, u\right)\frac{e^{u^{2}}}{u}\\
&&\times\sum_{c\in\c^{-1}\backslash\{0\}/\ringO_{F}^{\times +}}\norm(c\c)^{-1}\sum_{x\in \left(\a\Dif_{F}^{-1}\c^{-1}/\a\Dif_{F}^{-1}c\right)^{\times}}\exp\left(2\pi i\trace\left(\frac{\nu x}{c}\right)\right)\exp\left(2\pi i\trace\left(\frac{\xi\overline{x}}{c}\right)\right)\\
&&\times\int_{(\sigma)}\frac{\Gamma\left(\frac{\bk-1-s}{2}\right)}{\Gamma\left(\frac{\bk-1+s}{2}+1\right)}\left(\frac{2\pi\sqrt{\nu\xi[\a\b\c^{-2}]}}{|c|}\right)^{s}\;dsdu.
\end{eqnarray*}
It should be noted that multi-index notation is applied again in the integral representation of the $J$-Bessel function. Indeed, 
$\displaystyle{\int_{(\sigma)} \; ds}$ denotes the multiple integration $\displaystyle{\int_{(\sigma_1)}\cdots\int_{(\sigma_n)}}\; ds_n\cdots ds_1$.

Upon interchanging summations and integration, we get 
\begin{align*} 
E_{\g, \p,\a}(k)&=\int_{(3/2)}(2\pi)^{-2nu}\norm(\a)^{-u}\norm(\n\Dif_F^2)^{u}\gamma\left(\frac12,u\right)\frac{e^{u^{2}}}{u}\sum_{d=1}^{\infty}\frac{a_d(\n)}{d^{1+2u}}\\
&\hspace{.5in}\times\sum_{c\in\c^{-1}\backslash\{0\}/\ringO_{F}^{\times +}}\norm(c\c)^{-1}\sum_{x\in \left(\a\Dif_{F}^{-1}\c^{-1}/\a\Dif_{F}^{-1}c\right)^{\times}}\exp\left(2\pi i\trace\left(\frac{\xi\overline{x}}{c}\right)\right)\\
&\hspace{0.5in}\times\sum_{\nu\in(\a^{-1})^{+}}\frac{C_{\g}(\nu\a)}{\sqrt{\norm(\nu\a)}}\exp\left(2\pi i\trace\left(\frac{\nu x}{c}\right)\right)\\
&\hspace{.5in}\times \int_{(\sigma)}\frac{\Gamma\left(\frac{\bk-1-s}{2}\right)}{\Gamma\left(\frac{\bk-1+s}{2}+1\right)}\nu^{\frac{s}{2}-u}|c|^{-s}\xi^{\frac{s}{2}}[\a\b\c^{-2}]^{\frac{s}{2}}(2\pi)^{s}\;dsdu.
\end{align*}
For each $c\in\c^{-1}\backslash\{0\}/\ringO_{F}^{\times}$ and  $x\in \left(\a\Dif_{F}^{-1}\c^{-1}/\a\Dif_{F}^{-1}c\right)^{\times}$ in the above expression, we now focus on the inner sum over $\nu\in(\a^{-1})^+$, which we denote by $S(c,x,u)$. Using the relation between the (ad\`elic) Fourier coefficients $C_{\g}(\nu\a)$ and the (classical) Fourier coefficients $a_{\nu}(g_{\a})$ given in (\ref{eqn:classical-fourier-coefficients}), we can write
\[S(c,x,u)=\sum_{\nu\in(\a^{-1})^{+}}a_{\nu}(g_{\a})\exp\left(2\pi i\trace\left(\frac{\nu x}{c}\right)\right)\int_{(\sigma)}\frac{\Gamma\left(\frac{\bk-1-s}{2}\right)}{\Gamma\left(\frac{\bk-1+s}{2}+1\right)}\nu^{-\frac{l}{2}+\frac{s}{2}-u}|c|^{-s}\xi^{\frac{s}{2}}[\a\b\c^{-2}]^{\frac{s}{2}}(2\pi)^{s}\;ds.
\]

Before we proceed further, we define the classical Hilbert modular form $g_\a^{x,c}=(g_\a|\!|_l\alpha_{x,c})$ where $\alpha_{x,c}=\left[ {\begin{array}{cc}
   x & b \\
   c & \overline{x} \\
  \end{array} } \right]\in \SL_{2}(F)$ with $b$ being an element in $\c$ such that $x\overline{x}=1+bc$. 
This function has a Fourier expansion 
\[g_{\a}^{x,c}(z)=\sum_{\nu\in\left(\L_{\a}^{x,c}\right)^{+}} a_{\nu}(g_{\a}^{x,c})\exp\left(2\pi i\trace(\nu z)\right)\]
for some lattice $\L_{\a}^{x,c}$ in $F$. Moreover, it satisfies the following relation. 
 \begin{lemma}\label{lem:transformation}For all $y\in(\mathbb{R}^{+})^{n}$, we have
$\dis g_{\a}^{x,c}\left(\frac{i}{yc}-\frac{\overline{x}}{c}\right)=(iy)^{\bl}g_{\a}\left(\frac{iy}{c}+\frac{x}{c}\right).$
\end{lemma}

\noindent{\it Proof.}
\pushQED{\qed} 
This follows immediately upon taking $z=\frac{i}{yc}-\frac{\overline{x}}{c}$ in the identity 
\[g_{\a}^{x,c}(z)=(cz+\overline{x})^{-\bl}g_{\a}\left(\frac{xz+b}{cz+\overline{x}}\right). \hfill \qedhere
\popQED \]

Now that we have this notation set up, we go back to the computations on $S(c,x,u)$. Inserting the definition of the gamma function at $l/2-s/2+u$ gives
\begin{align*}
S(c,x,u)&=\sum_{\nu\in(\a^{-1})^{+}}a_{\nu}(g_{\a})\exp\left(2\pi i\trace\left(\frac{\nu x}{c}\right)\right)\int_{(\sigma)}\frac{|c|^{-s}\xi^{\frac{s}{2}}[\a\b\c^{-2}]^{\frac{s}{2}}(2\pi)^{s}\Gamma\left(\frac{\bk-1-s}{2}\right)}{\Gamma\left(\frac{\bk-1+s}{2}+1\right)\Gamma\left(\frac{\bl}{2}-\frac{s}{2}+u\right)}\left(\frac{c}{2\pi}\right)^{-\frac{\bl}{2}+\frac{s}{2}-u}\\
&\hspace{0.5in}\times\int_{F_{\infty}^{\times +}}y^{\frac{\bl}{2}-\frac{s}{2}+u-1}\exp\left(-2\pi\trace\left(\frac{\nu y}{c}\right)\right)\;dy\;ds\\
&=\int_{F_{\infty}^{\times+}}\sum_{\nu\in(\a^{-1})^{+}}a_{\nu}(g_{\a})\exp\left(-2\pi\trace\left(\frac{\nu y}{c}\right)\right)\exp\left(2\pi i\trace\left(\frac{\nu x}{c}\right)\right)\\
&\hspace{0.5in}\times\int_{(\sigma)}\frac{|c|^{-s}\xi^{\frac{s}{2}}[\a\b\c^{-2}]^{\frac{s}{2}}(2\pi)^{s}\Gamma\left(\frac{\bk-1-s}{2}\right)}{\Gamma\left(\frac{\bk-1+s}{2}+1\right)\Gamma\left(\frac{\bl}{2}-\frac{s}{2}+u\right)}\left(\frac{c}{2\pi}\right)^{-\frac{\bl}{2}+\frac{s}{2}-u}y^{\frac{\bl}{2}-\frac{s}{2}+u-1}\;ds\;dy.
\end{align*}
Notice that the sum over $\nu$ is nothing but the Fourier expansion of $g_\a$ at $z=\frac{x}{c}+i\frac{y}{c}$, and therefore we have 
\begin{align*}
S(c,x,u)&=\int_{F_{\infty}^{\times+}}g_{\a}\left(\frac{x}{c}+i\frac{y}{c}\right)\int_{(\sigma)}\frac{|c|^{-s}\xi^{\frac{s}{2}}[\a\b\c^{-2}]^{\frac{s}{2}}(2\pi)^{s}\Gamma\left(\frac{\bk-1-s}{2}\right)}{\Gamma\left(\frac{\bk-1+s}{2}+1\right)\Gamma\left(\frac{\bl}{2}-\frac{s}{2}+u\right)}\left(\frac{c}{2\pi}\right)^{-\frac{\bl}{2}+\frac{s}{2}-u}y^{\frac{\bl}{2}-\frac{s}{2}+u-1}\;ds\;dy\\
&=i^{-\bl}\int_{F_{\infty}^{\times+}}g_{\a}^{x,c}\left(\frac{i}{yc}-\frac{\overline{x}}{c}\right)\int_{(\sigma)}\frac{|c|^{-s}\xi^{\frac{s}{2}}[\a\b\c^{-2}]^{\frac{s}{2}}(2\pi)^{s}\Gamma\left(\frac{\bk-1-s}{2}\right)}{\Gamma\left(\frac{\bk-1+s}{2}+1\right)\Gamma\left(\frac{\bl}{2}-\frac{s}{2}+u\right)}\\
&\hspace{0.5in}\times\left(\frac{c}{2\pi}\right)^{-\frac{\bl}{2}+\frac{s}{2}-u}y^{-\frac{\bl}{2}-\frac{s}{2}+u-1}\;ds\;dy, 
\end{align*}
where the last equality follows from Lemma \ref{lem:transformation}. Hence, we get 
\begin{align*}
S(c,x,u)&=i^{-\bl}\int_{F_{\infty}^{\times+}}\sum_{\nu\in\left(\L_{\a}^{x,c}\right)^{+}} a_{\nu}(g_{\a}^{x,c})\exp\left(2\pi i\trace\left(-\frac{\nu\overline{x}}{c}\right)\right)\exp\left(-2\pi \trace\left(\frac{\nu}{yc}\right)\right)\\
&\hspace{0.5in}\times\int_{(\sigma)}\frac{|c|^{-s}\xi^{\frac{s}{2}}[\a\b\c^{-2}]^{\frac{s}{2}}(2\pi)^{s}\Gamma\left(\frac{\bk-1-s}{2}\right)}{\Gamma\left(\frac{\bk-1+s}{2}+1\right)\Gamma\left(\frac{\bl}{2}-\frac{s}{2}+u\right)}\left(\frac{c}{2\pi}\right)^{-\frac{\bl}{2}+\frac{s}{2}-u}y^{-\frac{\bl}{2}-\frac{s}{2}+u-1}\;ds\;dy\\
&=i^{-\bl}\sum_{\nu\in\left(\L_{\a}^{x,c}\right)^{+}} a_{\nu}(g_{\a}^{x,c})\exp\left(2\pi i\trace\left(-\frac{\nu\overline{x}}{c}\right)\right)\int_{(\sigma)}\frac{|c|^{-s}\xi^{\frac{s}{2}}[\a\b\c^{-2}]^{\frac{s}{2}}(2\pi)^{s}\Gamma\left(\frac{\bk-1-s}{2}\right)}{\Gamma\left(\frac{\bk-1+s}{2}+1\right)\Gamma\left(\frac{\bl}{2}-\frac{s}{2}+u\right)}\\
&\hspace{0.5in}\times\left(\frac{c}{2\pi}\right)^{-\frac{\bl}{2}+\frac{s}{2}-u}\int_{F_{\infty}^{\times+}}y^{-\frac{\bl}{2}-\frac{s}{2}+u-1}\exp\left(-2\pi \trace\left(\frac{\nu}{yc}\right)\right)\;dy\;ds.
\end{align*}
Applying the integral definition of the gamma function once again results in
\begin{align*}
S(c, x, u)&=i^{-\bl}\sum_{\nu\in\left(\L_{\a}^{x,c}\right)^{+}} a_{\nu}(g_{\a}^{x,c})\exp\left(2\pi i\trace\left(-\frac{\nu\overline{x}}{c}\right)\right)\\
&\hspace{.2in}\times\int_{(\sigma)}\frac{|c|^{-s}\xi^{\frac{s}{2}}[\a\b\c^{-2}]^{\frac{s}{2}}(2\pi)^{s}\Gamma\left(\frac{\bk-1-s}{2}\right)\Gamma\left(\frac{\bl}{2}+\frac{s}{2}-u\right)}{\Gamma\left(\frac{\bk-1+s}{2}+1\right)\Gamma\left(\frac{\bl}{2}-\frac{s}{2}+u\right)}\left(\frac{c}{2\pi}\right)^{s-2u}\nu^{-\frac{\bl}{2}-\frac{s}{2}+u}\;ds.
\end{align*}
By folding the $\nu$-sum into a sum over $\left(\L_{\a}^{x,c}\right)^{+}/\ringO_{F}^{\times+}$, which simply amounts to re-introducing the $\eta$-sum, we can write $S(c,x,u)$ as 
\begin{align*}
S(c,x,u)&=i^{-\bl}(2\pi)^{2u}\sum_{\eta\in\ringO_{F}^{\times+}}\int_{(\sigma)}\frac{|c|^{-s}\xi^{\frac{s}{2}}[\a\b\c^{-2}]^{\frac{s}{2}}\Gamma\left(\frac{\bk-1-s}{2}\right)\Gamma\left(\frac{\bl}{2}+\frac{s}{2}-u\right)}{\Gamma\left(\frac{\bk-1+s}{2}+1\right)\Gamma\left(\frac{\bl}{2}-\frac{s}{2}+u\right)}c^{s-2u}\\
&\hspace{.5in}\times\sum_{\nu\in\left(\L_{\a}^{x,c}\right)^{+}/\ringO_{F}^{\times+}} a_{\nu\eta}(g_{\a}^{x,c})(\nu\eta)^{-\frac{\bl}{2}-\frac{s}{2}+u}\exp\left(2\pi i\trace\left(-\frac{\nu\eta\overline{x}}{c}\right)\right)\;ds\\
&=i^{-\bl}(2\pi)^{2u}\sum_{\eta\in\ringO_{F}^{\times+}}\int_{(\sigma)}\frac{|c|^{-s}\xi^{\frac{s}{2}}[\a\b\c^{-2}]^{\frac{s}{2}}\Gamma\left(\frac{\bk-1-s}{2}\right)\Gamma\left(\frac{\bl}{2}+\frac{s}{2}-u\right)}{\Gamma\left(\frac{\bk-1+s}{2}+1\right)\Gamma\left(\frac{\bl}{2}-\frac{s}{2}+u\right)}c^{s-2u}\eta^{-\frac{s}{2}}\\
&\hspace{.5in}\times\sum_{\nu\in\left(\L_{\a}^{x,c}\right)^{+}/\ringO_{F}^{\times+}} a_{\nu}(g_{\a}^{x,c})\nu^{-\frac{l}{2}-\frac{s}{2}+u}\exp\left(2\pi i\trace\left(-\frac{\nu\eta\overline{x}}{c}\right)\right)\;ds.
\end{align*} 
 We move the line of integration in $s$ to $\Re(s_{j})=6+{\lambda_{j}}$ with certain choices of $\lambda_{j}\geq0$. Indeed, if $\eta_{j}\leq1$, we choose $\lambda_{j}=0$, and if $\eta_{j}>1$, we choose $\lambda_{j}=\lambda_0$ for some fixed positive constant $\lambda_0$. Such choices of $\lambda_{j}$ are made to ensure that the sum over $\eta$ in the error term is absolutely convergent. This is mainly guaranteed by virtue of the crucial fact (see Luo \cite[page~136]
{luo1})
\begin{equation}\label{eqn:luo}
\sum_{\eta\in\ringO_{F}^{\times+}}\prod_{|\eta_j|>1}|\eta_j|^{-\lambda_0}<\infty.\end{equation}  
Upon substituting $S(c,x,u)$ back in $E_{\g,\p,\a}(k)$, we get
\begin{align*}E_{\g,\p,\a}(k)&=i^{-l}\int_{(3/2)}\norm(\c)^{-1}\norm(\a)^{-u}\norm(\n\Dif_F^2)^{u}\gamma\left(\frac12,u\right)\frac{e^{u^{2}}}{u}\sum_{d=1}^{\infty}\frac{a_d(\n)}{d^{1+2u}}\\
&\hspace{.5in}\times\sum_{c\in\c^{-1}\backslash\{0\}/\ringO_{F}^{\times +}}\sum_{x\in \left(\a\Dif_{F}^{-1}\c^{-1}/\a\Dif_{F}^{-1}c\right)^{\times}}\exp\left(2\pi i\trace\left(\frac{\xi\overline{x}}{c}\right)\right)\\
&\hspace{0.5in}\times \sum_{\eta\in\ringO_{F}^{\times+}}\int_{(\sigma)}\frac{|c|^{-s-1}\xi^{\frac{s}{2}}[\a\b\c^{-2}]^{\frac{s}{2}}\Gamma\left(\frac{\bk-1-s}{2}\right)\Gamma\left(\frac{\bl}{2}+\frac{s}{2}-u\right)}{\Gamma\left(\frac{\bk-1+s}{2}+1\right)\Gamma\left(\frac{\bl}{2}-\frac{s}{2}+u\right)}c^{s-2u}\eta^{-\frac{s}{2}}\\
&\hspace{.5in}\times\sum_{\nu\in\left(\L_{\a}^{x,c}\right)^{+}/\ringO_{F}^{\times+}} a_{\nu}(g_{\a}^{x,c})\nu^{-\frac{l}{2}-\frac{s}{2}+u}\exp\left(2\pi i\trace\left(-\frac{\nu\eta\overline{x}}{c}\right)\right)\;ds\;du.\end{align*}
 Interchanging summation and integration once more allows us to write
\begin{align*}\label{eqn:error_a}
E_{\g,\p,\a}(k)&=i^{-\bl}\int_{(3/2)}\norm(\c)^{-1}\norm(\a)^{-u}\norm(\n\Dif_F^2)^{u}\gamma\left(\frac12, u\right)\frac{e^{u^{2}}}{u}\sum_{d=1}^{\infty}\frac{a_d(\n)}{d^{1+2u}}\\
&\hspace{0.5in}\times \sum_{\eta\in\ringO_{F}^{\times+}}\int_{(\sigma)}\frac{\xi^{\frac{s}{2}}[\a\b\c^{-2}]^{\frac{s}{2}}\Gamma\left(\frac{\bk-1-s}{2}\right)\Gamma\left(\frac{\bl}{2}+\frac{s}{2}-u\right)}{\Gamma\left(\frac{\bk-1+s}{2}+1\right)\Gamma\left(\frac{\bl}{2}-\frac{s}{2}+u\right)}\eta^{-\frac{s}{2}}\nonumber \\
&\hspace{.5in}\times\sum_{c\in\c^{-1}\backslash\{0\}/\ringO_{F}^{\times +}}\sum_{x\in \left(\a\Dif_{F}^{-1}\c^{-1}/\a\Dif_{F}^{-1}c\right)^{\times}}|c|^{-s-1}c^{s-2u}\nonumber \\
&\hspace{.5in}\times\sum_{\nu\in\left(\L_{\a}^{x,c}\right)^{+}/\ringO_{F}^{\times+}} a_{\nu}(g_{\a}^{x,c})\nu^{-\frac{l}{2}-\frac{s}{2}+u}\exp\left(2\pi i\trace\left(\frac{(\xi-\nu\eta)\overline{x}}{c}\right)\right)\;ds\;du. \nonumber
\end{align*}
Let us now examine the multiple sum (over $c$, $x$, and $\nu$),
 \begin{align*}S(u,s)&=\sum_{c\in\c^{-1}\backslash\{0\}/\ringO_{F}^{\times +}}\sum_{x\in \left(\a\Dif_{F}^{-1}\c^{-1}/\a\Dif_{F}^{-1}c\right)^{\times}}|c|^{-s-1}c^{s-2u}\\
 &\hspace{.5in}\times\sum_{\nu\in\left(\L_{\a}^{x,c}\right)^{+}/\ringO_{F}^{\times+}} a_{\nu}(g_{\a}^{x,c})\nu^{-\frac{l}{2}-\frac{s}{2}+u}\exp\left(2\pi i\trace\left(\frac{(\xi-\nu\eta)\overline{x}}{c}\right)\right).
 \end{align*}
 
\begin{lemma} 
Let $u$ and $s$ be as above. Then, $S(u,s)$ is absolutely convergent.
\end{lemma} 
\begin{proof}
Let us fix a pair $(c_0,x_0)$ with 
$c_0\in\c^{-1}\backslash\{0\}/\ringO_{F}^{\times}$ and  $x_0\in\left(\a\Dif_{F}^{-1}\c^{-1}/\a\Dif_{F}^{-1}c\right)^{\times}$. This gives rise to the fixed matrix $\alpha:=\alpha_{x_0,c_0}\in \SL_{2}(F)$. In fact, it can be verified that $\alpha_{x,c}\alpha^{-1}$ is in $\Gamma_0(\ringO_{F},\a\Dif_{F}^{-1})$ for any $c\in\c^{-1}\backslash\{0\}/\ringO_{F}^{\times}$ and $x\in\left(\a\Dif_{F}^{-1}\c^{-1}/\a\Dif_{F}^{-1}c\right)^{\times}$. Notice that since the congruence subgroup $\Gamma_0(\n,\a\Dif_{F}^{-1})$ has finite index in $\Gamma_0(\ringO_{F},\a\Dif_{F}^{-1})$, we have 
$$\Gamma_0(\ringO_{F},\a\Dif_{F}^{-1})=\cup_{i=1}^{r}\Gamma_0(\n,\a\Dif_{F}^{-1})\delta_{i},$$ 
for some finite number $r$ and matrices $\delta_{1},\cdots,\delta_{r}\in\Gamma_0(\ringO_{F},\a\Dif_{F}^{-1})$. It follows that there exist $\gamma\in\Gamma_0(\n,\a\Dif_{F}^{-1})$ and $\delta\in\{\delta_{1},\cdots,\delta_{r}\}$ such that $\alpha_{x,c}=\gamma\delta\alpha$. Hence, we have 
$$g_{\a}^{x,c}=g_{\a}|\!|_{\bl}{\alpha_{x,c}}=g_{\a}|\!|_{\bl}{\gamma\delta\alpha}=g_{\a}|\!|_{\bl}{\delta\alpha}.$$ 
If we set $g_{\a}^{\delta}:=g_{\a}|\!|_{\bl}{\delta}$, we get $g_{\a}^{x,c}=g_{\a}^{\delta}|\!|_{\bl}{\alpha}$. Therefore, 
\begin{align*}
S(u,s)&\ll\sum_{c\in\c^{-1}\backslash\{0\}/\ringO_{F}^{\times +}}\sum_{x\in \left(\a\Dif_{F}^{-1}\c^{-1}/\a\Dif_{F}^{-1}c\right)^{\times}}|c|^{-2\Re(u)-1}\sum_{\nu\in\left(\L_{\a}^{x,c}\right)^{+}/\ringO_{F}^{\times +}}\sum_{i=1}^{r} \left|a_{\nu}\left(g_{\a}^{\delta_{i}}|\!|_{\bl} \alpha\right)\right|\nu^{-\frac{l}{2}-\frac{\Re(s)}{2}+\Re(u)}\\
 &\ll\sum_{c\in\c^{-1}\backslash\{0\}/\ringO_{F}^{\times +}}|c|^{-3}\sum_{\nu\in\left(\L_{\a}^{x,c}\right)^{+}/\ringO_{F}^{\times+}}\nu^{-\frac{3}{2}}\\
 &\ll1.  
 \end{align*}
In the above inequality, we applied the crude estimate $a_\nu=O(\nu^{l/2})$.
\end{proof}
We write $u=3/2+iv$, $s_{j}=6+\lambda_{j}+it_{j}$. For ease of notation, we put $dt=dt_{1}\cdots dt_{n}$ and $\lambda=(\lambda_1,\dots,\lambda_n)$. Recall that we set $G(u)=e^{u^2}$ in which case $|G(u)|=e^{\Re(u^2)}\ll e^{-v^2}$.
Hence, 
\begin{align*}
E_{\g,\p,\a}(k)&\ll\int_{-\infty}^{\infty}\frac{e^{-v^2}}{\sqrt{\frac{9}{4}+v^2}}\left|\frac{\Gamma\left(\frac{3}{2}+\frac{k-l+1}{2}+iv\right)\Gamma\left(\frac{3}{2}+\frac{k+l-1}{2}+iv\right)}{\Gamma\left(\frac{k-l+1}{2}\right)\Gamma\left(\frac{k+l-1}{2}\right)}\right|\\
&\hspace{.5in}\times\sum_{\eta\in\ringO_{F}^{\times+}}\eta^{-6-\lambda}\int_{-\infty}^{\infty}\left|\frac{\Gamma\left(\frac{k-7-\lambda}{2}-\frac{it}{2}\right)\Gamma\left(\frac{3+\lambda+l}{2}+\frac{it}{2}-iv\right)}{\Gamma\left(\frac{k+7+\lambda}{2}+\frac{it}{2}\right)\Gamma\left(\frac{-3-\lambda+l}{2}-\frac{it}{2}+iv\right)}\right|\;dtdv.
\end{align*}
Next, we observe that 
\begin{align*}
&\sum_{\eta\in\ringO_{F}^{\times+}}\eta^{-6-\lambda}\int_{-\infty}^{\infty}\left|\frac{\Gamma\left(\frac{k-7-\lambda}{2}-\frac{it}{2}\right)\Gamma\left(\frac{3+\lambda+l}{2}+\frac{it}{2}-iv\right)}{\Gamma\left(\frac{k+7+\lambda}{2}+\frac{it}{2}\right)\Gamma\left(\frac{-3-\lambda+l}{2}-\frac{it}{2}+iv\right)}\right|\;dt\\
&\hspace{.4in}=\sum_{\eta\in\ringO_{F}^{\times+}}\prod_{\substack{j=1\\ \eta_{j}>1}}^{n}\eta_{j}^{-\lambda_0}\int_{-\infty}^{\infty}\left|\frac{\Gamma\left(\frac{k_{j}-7-\lambda_0}{2}-\frac{it_{j}}{2}\right)\Gamma\left(\frac{3+\lambda_0+l_{j}}{2}+\frac{it_{j}}{2}-iv\right)}{\Gamma\left(\frac{k_{j}+7+\lambda_0}{2}+\frac{it_{j}}{2}\right)\Gamma\left(\frac{-3-\lambda_0+l_{j}}{2}-\frac{it_{j}}{2}+iv\right)}\right|\;dt_{j}\\
&\hspace{.9in}\times\prod_{\substack{j=1\\\eta_{j}\leq1}}^{n}\int_{-\infty}^{\infty}\left|\frac{\Gamma\left(\frac{k_{j}-7}{2}-\frac{it_{j}}{2}\right)\Gamma\left(\frac{3+l_j}{2}+\frac{it_{j}}{2}-iv\right)}{\Gamma\left(\frac{k_j+7}{2}+\frac{it_{j}}{2}\right)\Gamma\left(\frac{-3+l_j}{2}-\frac{it_{j}}{2}+iv\right)}\right|\;dt_{j}.\end{align*}
Let us now consider the integral 
\begin{align*}
I_{j}(v)&=\int_{-\infty}^{\infty}\left|\frac{\Gamma\left(\frac{k_{j}-7-\lambda_j}{2}-\frac{it_{j}}{2}\right)\Gamma\left(\frac{3+\lambda_j+l_{j}}{2}-i(v-\frac{t_{j}}{2})\right)}{\Gamma\left(\frac{k_{j}+7+\lambda_j}{2}+\frac{it_{j}}{2}\right)\Gamma\left(\frac{-3-\lambda_j+l_{j}}{2}+i(v-\frac{t_{j}}{2})\right)}\right|\;dt_{j}\\
&=\int_{-\infty}^{\infty}\left|\frac{\Gamma\left(\frac{k_{j}+7+\lambda_j}{2}-7-\lambda_j+\frac{it_{j}}{2}\right)\Gamma\left(\frac{-3-\lambda_j+l_{j}}{2}+3+\lambda_j+i(v-\frac{t_{j}}{2})\right)}{\Gamma\left(\frac{k_{j}+7+\lambda_j}{2}+\frac{it_{j}}{2}\right)\Gamma\left(\frac{-3-\lambda_j+l_{j}}{2}+i(v-\frac{t_{j}}{2})\right)}\right|\;dt_j.
\end{align*}
By Lemma~\ref{gamma-quotient}, we have
\begin{align*}
I_j(v)&\ll\int_{-\infty}^{\infty}\left|\frac{k_{j}+7+\lambda_j}{2}+\frac{it_{j}}{2}\right|^{-7-\lambda_j}\left|\frac{-3-\lambda_j+l_{j}}{2}+i\left(v-\frac{t_{j}}{2}\right)\right|^{3+\lambda_j}\;dt_j\\
&\ll\int_{-\infty}^{\infty}(k_{j}^{2}+t_{j}^{2})^{\frac{-7-\lambda_j}{2}}\left(l_j^2+\left(v-t_{j}\right)^2\right)^{\frac{3+\lambda_j}{2}}\;dt_j.
\end{align*}
Furthermore, the application of the change of variable, $t_{j}=k_{j}\tan(\theta_{j})$, yields \begin{align*}I_{j}(v)&\ll k_{j}^{-6-\lambda_j}\int_{-\frac{\pi}{2}}^{\frac{\pi}{2}}|\cos(\theta_{j})|^{5+\lambda_j}\left(v^2+k_{j}^2\tan^{2}(\theta_{j})\right)^{\frac{3+\lambda_j}{2}}\;d\theta_{j}\\&\ll k_{j}^{-6-\lambda_j}\left(v^2+k_{j}^2\right)^{\frac{3+\lambda_j}{2}}.\end{align*} 
Therefore, 
\begin{align*}
&\sum_{\eta\in\ringO_{F}^{\times+}}\eta^{-6-\lambda}\int_{-\infty}^{\infty}\left|\frac{\Gamma\left(\frac{k-7-\lambda}{2}-\frac{it}{2}\right)\Gamma\left(\frac{3+\lambda+l}{2}+\frac{it}{2}-iv\right)}{\Gamma\left(\frac{k+7+\lambda}{2}+\frac{it}{2}\right)\Gamma\left(\frac{-3-\lambda+l}{2}-\frac{it}{2}+iv\right)}\right|\;dt\\
&\hspace{.5in}\ll\sum_{\eta\in\ringO_{F}^{\times+}}\prod_{\substack{j=1\\\eta_{j}>1}}^{n}\eta_{j}^{-\lambda_0}k_{j}^{-6-\lambda_0}\left(v^2+k_{j}^2\right)^{\frac{3+\lambda_0}{2}}\prod_{\substack{j=1\\\eta_{j}\leq1}}^{n}k_{j}^{-6}\left(v^2+k_{j}^2\right)^{\frac{3}{2}}\\
&\hspace{.5in}\ll\sum_{\eta\in\ringO_{F}^{\times+}}\prod_{\substack{j=1\\ \eta_j>1}}^{n}\eta_{j}^{-\lambda_0}k_{j}^{-\lambda_0}\left(v^2+k_{j}^2\right)^{\frac{\lambda_0}{2}}\prod_{j=1}^{n}k_{j}^{-6}\left(v^2+k_{j}^2\right)^{\frac{3}{2}}\\
&\hspace{.5in}\ll|v|^{n\lambda_0}k^{-6}\left(v^2+k^2\right)^{\frac{3}{2}}.
\end{align*} 
Notice that the last inequality is guaranteed by (\ref{eqn:luo}). Finally, since 
\[\left|\frac{\Gamma\left(\frac{3}{2}+\frac{\bk-{\bl}+1}{2}+iv\right)\Gamma\left(\frac{3}{2}+\frac{\bk+\bl-1}{2}+iv\right)}{\Gamma\left(\frac{\bk-\bl+1}{2}\right)\Gamma\left(\frac{\bk+\bl-1}{2}\right)}\right|\ll\left(k^{2}+v^2\right)^{\frac{3}{2}},\] 
we conclude that
\begin{align*}
E_{\g,\p,\a}(k)&\ll\int_{-\infty}^{\infty}\frac{|v|^{n\lambda_0}e^{-v^2}}{\sqrt{\frac{9}{4}+v^2}}\left(k^{2}+v^2\right)^{3}k^{-6}\;dv\\
&\ll\int_{0}^{\infty}\frac{v^{n(\lambda_0+6)}e^{-v^2}}{\sqrt{\frac{9}{4}+v^2}}\;dv\\
&\ll1.\end{align*}
This proves the second statement of Lemma~\ref{lem:main}.

\section*{Acknowledgements}
The authors would like to thank Amir Akbary for useful remarks which improved the exposition of the manuscript. The authors would also like to extend their gratitude to Jeff Hoffstein for valuable discussions about the topic of this paper.

\bibliographystyle{siam}
\bibliography{references}

\end{document}